\documentclass[11pt]{article}
\usepackage{amssymb, amsmath, hyperref, array, color, bm}
\usepackage{tikz}
\usepackage{mathdots}

\newcommand{\hide}[1]{}

\newtheorem{defn}{Definition}
\newtheorem{notn}[defn]{Notation}
\newtheorem{lemma}[defn]{Lemma}
\newtheorem{proposition}[defn]{Proposition}
\newtheorem{theorem}[defn]{Theorem}
\newtheorem{corollary}[defn]{Corollary}
\newtheorem{remark}[defn]{Remark}
\newtheorem{example}[defn]{Example}

           {\vspace{3.3mm}
           \noindent{\bf #1}\it}%
           {\vspace{3.3mm}}

\newenvironment{proof}[1]{
  \trivlist \item[\hskip \labelsep{\it #1}]}{\hfill\mbox{$\square$}
  \endtrivlist}

\def\Z{\mathbb{Z}}

\def\R{{\rm{\bf R}}}

\def\D{{\rm{\bf D}}}

\def\sRes{{\rm sRes}}
\def\sResP{{\rm sResP}}

\def\Ind{{\rm{Ind}}}
\def\sign{{\rm{sign}}}

\newcommand {\W}      {\mathrm{MVar}}

\newcommand {\SR}      {\mathrm{SResP}}

\date{}

\title{A new general formula 
for the Cauchy Index on an interval with Subresultants}

\author{Daniel Perrucci$^{\flat}$\thanks{{\scriptsize Partially supported by the Argentinian grants} {\footnotesize UBACYT 20020160100039BA} 
{\scriptsize and} {\footnotesize PIP 11220130100527CO CO\-NI\-CET}
}  \quad  Marie-Fran\c{c}oise Roy$^{\sharp}$ \\[5mm]
{\small ${\flat}$ Departamento de Matem\'atica, FCEN, Universidad de Buenos Aires
and IMAS UBA-CONICET,}\\
{\small Ciudad Universitaria, 1428 Buenos Aires, Argentina}\\ 
{\small ${\sharp}$ IRMAR (UMR CNRS 6625), Universit\'e de Rennes 1,} \\
{\small Campus de Beaulieu, 35042 Rennes Cedex,  France}}

\begin{document}

\maketitle

\begin{abstract}

We present a new formula 
for the Cauchy index of a rational function
on an interval using subresultant polynomials.
There is no condition on the endpoints of the interval and the formula also involves in some cases
less subresultant polynomials. 
\end{abstract}

\bigskip


\medskip


\section{Introduction}
\label{sec:introd}

Let $(\R, \le)$ be a real closed field and $P, Q \in \R[X]$ with $P \ne 0$. 
Already considered by Sturm and Cauchy (\cite{Stu,Cau}), the Cauchy index of the rational function $\displaystyle{\frac{Q}P}$ is the integer number which counts its number of \emph{jumps}  
from $-\infty$ to 
$+\infty$ minus its number of \emph{jumps} from $+\infty$ to $-\infty$. This value
plays an important role in many algorithms in real algebraic geometry (\cite{Barnett, BPRbook,  Henrici}). 
For instance, 
the Tarski query of $Q$ for $P$, defined as
$$
\hbox{TaQ}(Q, P)
=
\# \Big\{x \in \R \ | \ P(x) = 0, \ Q(x) > 0 \Big\} \ - \ 
\# \Big\{x \in \R \ | \ P(x) = 0, \ Q(x) < 0 \Big\},
$$  is equal to the Cauchy index of the rational function $\displaystyle{\frac{P'Q}P}$
(see, for instance, \cite[Proposition 2.57]{BPRbook}). Tarski queries are used to 
solve 
the 
\emph{sign determination problem}, which consists in 
listing
the signs of a list of polynomials in $\R[X]$ evaluated
at the roots in $\R$ of another polynomial in $\R[X]$ (see \cite[Section 10.3]{BPRbook}). 
In particular, 
the 
number of real roots of a polynomial $P \in \R[X]\setminus\{0\}$
coincides with 
the Tarski query 
$\hbox{TaQ}(1, P)$ and is equal to 
the Cauchy index of the rational function $\displaystyle{\frac{P'}P}$.
We can also mention the role of the Cauchy index for complex roots counting (see  \cite{Cau,Eis,PerRoy}).
For more information and references to the history of the Cauchy index see \cite{Eis}.

\subsection{Cauchy index}

Let $P, Q \in \R[X]$ with $P \ne 0$.  The usual definition of the Cauchy index of $\displaystyle{\frac{Q}P}$ 
is  made directly on intervals whose endpoints  are not roots of $P$. In this paper 
we use the extended definition of the Cauchy index 
introduced in \cite[Section 3]{Eis}, which is
made first locally at elements in $\R$,  and 
then on intervals without restriction.

\begin{defn}\label{defn:CI_at_a_point} 
Let $x \in \R$ and $P, Q \in \R[X], P\not=0$.
\begin{itemize}

\item  If $Q\not=0$,
the rational function $\displaystyle{\frac{Q}P}$ can be written uniquely as
$$
\frac{Q}{P} = (X - x)^{m}\frac{\widetilde Q}{\widetilde P}
$$
with $m\in \Z$, $\widetilde P, \widetilde Q \in \R[X]$, $\widetilde P$ monic, 
$\widetilde P$ and $\widetilde Q$                                                 
coprime                                                 
and  $\widetilde P(x) \ne 0, \widetilde Q(x) \ne 0$. 
For $\varepsilon \in \{+, -\}$, define
$$
{\rm Ind}_x^{\varepsilon} \Big( \frac{Q}{P} \Big) = \left\{
\begin{array}{ll}
\frac{1}2 \cdot \sign\Big( 
\displaystyle{\frac{\widetilde{Q}(x)}{\widetilde{P}(x)}} \Big)  & 
\hbox{if } \varepsilon = + \hbox{ and } m<0,\\[3mm]
\frac{1}2 \cdot (-1)^{m} \cdot \sign\Big( 
\displaystyle{\frac{\widetilde{Q}(x)}{\widetilde{P}(x)}} \Big)  & 
\hbox{if } \varepsilon = - \hbox{ and } m<0.
 \end{array}
\right.
$$
In all other cases, define
$$
{\rm Ind}_x^{\varepsilon} \Big( \frac{Q}{P} \Big) = 0
$$

\item The Cauchy index of $ \displaystyle{\frac{Q}P}$ at $x$ is 
$$
{\rm Ind}_x\Big( \frac{Q}{P} \Big) = 
{\rm Ind}_x^{+} \Big( \frac{Q}{P} \Big) 
-
{\rm Ind}_x^{-} \Big( \frac{Q}{P} \Big).
$$
\end{itemize}
\end{defn}

Said in other terms, when $x$ is a pole of $\displaystyle{\frac{Q}P}$, 
we have that $\displaystyle{{\rm Ind}_x^{+} \Big( \frac{Q}{P} \Big)}$ is one half of the sign of 
$\displaystyle{\frac{Q}P}$ \emph{to the right} of $x$, and 
$\displaystyle{{\rm Ind}_x^{-} \Big( \frac{Q}{P} \Big)}$ is one half of the sign of 
$\displaystyle{\frac{Q}P}$ \emph{to the     left} of $x$. 
Then, the Cauchy index of $\displaystyle{\frac{Q}P}$ at $x$, $\displaystyle{{\rm Ind}_x \Big( \frac{Q}{P} \Big)}$, 
is simply the difference between them. 
We illustrate this notion considering the graph of the function
$\displaystyle{\frac{Q}P}$ around $x$ in each different case.

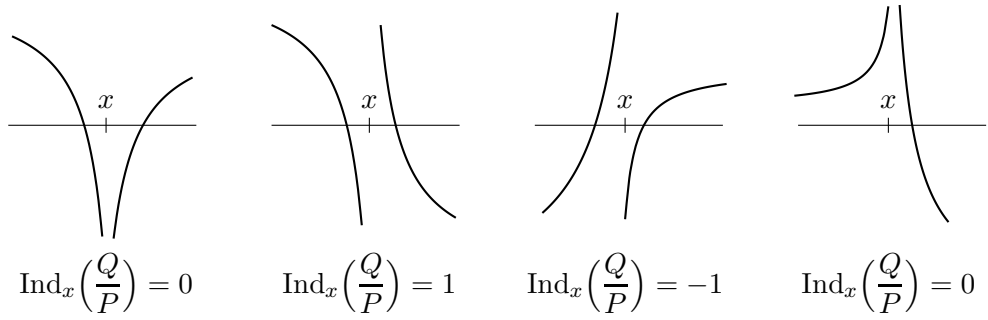
\begin{center}
\begin{tikzpicture}
      \draw[-] (-6.5,0) -- (-4,0);
      \draw[-] (-3,0) -- (-0.5,0);
      \draw[-] (0.5,0) -- (3,0);
      \draw[-] (4,0) -- (6.5,0);
      \draw[-] (-5.2,-0.1) -- (-5.2,0.1) node[above] {$x$} ;
      \draw[-] (-1.7,-0.1) -- (-1.7,0.1) node[above] {$x$} ;
      \draw[-] (1.7,-0.1) -- (1.7,0.1) node[above] {$x$} ;
      \draw[-] (5.2,-0.1) -- (5.2,0.1) node[above] {$x$} ;
      \draw[line width=0.8pt, domain=-6.5:-5.3,smooth,variable=\x] plot ({\x+0.05},{1/(\x+5)+1.85});
      \draw[line width=0.8pt, domain=-5.05:-4,smooth,variable=\x,] plot ({\x-0.05},{-1/(\x+5.4)+1.35});
      \draw[line width=0.8pt, domain=-3:-1.8,smooth,variable=\x] plot ({\x},{1/(\x+1.5)+2});
      \draw[line width=0.8pt, domain=-1.5:-0.5,smooth,variable=\x] plot ({\x-0.05},{1/(\x+1.8)-2});
      \draw[line width=0.8pt, domain=0.5:1.5,smooth,variable=\x] plot ({\x+0.1},{-2/(\x-2)-2.5});
      \draw[line width=0.8pt, domain=1.65:3,smooth,variable=\x] plot ({\x+0.05},{-0.3/(\x-1.5)+0.75});
      \draw[line width=0.8pt, domain=4:5.25,smooth,variable=\x,] plot ({\x-0.05},{-0.2/(\x-5.4)+0.25});
      \draw[line width=0.8pt, domain=5.25:5.9,smooth,variable=\x] plot ({\x+0.1},{1/(\x-5)-2.4}); 
      \node at (-5.2,-2.1) {$\Ind_x\Big(\displaystyle{\frac{Q}P}\Big) = 0 $};
      \node at (-1.7,-2.1) {$\Ind_x\Big(\displaystyle{\frac{Q}P}\Big) = 1 $};
      \node at (1.7,-2.1) {$\Ind_x\Big(\displaystyle{\frac{Q}P}\Big) = -1 $};
      \node at (5.2,-2.1) {$\Ind_x\Big(\displaystyle{\frac{Q}P}\Big) = 0 $};
      
      \end{tikzpicture}
\end{center}

\begin{defn} Let $a, b \in \R$ with $a < b$ and $P, Q \in \R[X]$ with $P \ne 0$. 
If $Q \ne 0$, the Cauchy index of $ \displaystyle{\frac{Q}P }$ on $[a, b]$ is
$$
{\rm Ind}_a^b\Big( \frac{Q}{P} \Big) = 
{\rm Ind}_a^+\Big( \frac{Q}{P} \Big)  + \sum_{x \in (a,b) }  {\rm Ind}_x \Big( \frac{Q}{P} \Big) 
- {\rm Ind}_b^-\Big( \frac{Q}{P} \Big),
$$
where the sum is well-defined since only roots $x$ of $P$ in $(a, b)$ contribute. 

Similarly, 
if $Q \ne 0$, 
the Cauchy index of 
$ \displaystyle{\frac{Q}P }$ on $\R$ is
$$
{\rm Ind}_{\R}\Big( \frac{Q}{P} \Big) = 
\sum_{x \in \R }  {\rm Ind}_x \Big( \frac{Q}{P} \Big)
$$
where, again, the sum is well-defined since only roots $x$ of $P$ in $\R$ contribute. 

If $Q = 0$, both
the Cauchy index of $ \displaystyle{\frac{Q}P }$ on $[a, b]$ 
and 
the Cauchy index of $ \displaystyle{\frac{Q}P }$ on $\R$ 
are defined as $0$. 
\end{defn}

In the following picture we consider again the graph of the function
$\displaystyle{\frac{Q}P}$, this time in $[a, b]$.

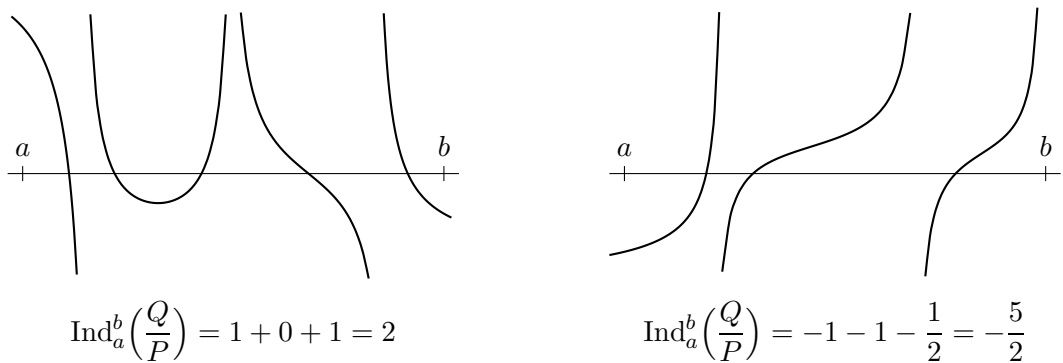
\begin{center}
\begin{tikzpicture}
      \draw[-] (-7,0) -- (-1,0);
      \draw[-] (1,0) -- (7,0);
      \draw[-] (-6.8,-0.1) -- (-6.8,0.1) node[above] {$a$};
      \draw[-] (-1.2,-0.1) -- (-1.2,0.1) node[above] {$b$};
      \draw[-] (1.2,-0.1) -- (1.2,0.1) node[above] {$a$};
      \draw[-] (6.8,-0.1) -- (6.8,0.1) node[above] {$b$} ;
      \draw[line width=0.8pt, domain=-7:-6.13,smooth,variable=\x] plot ({\x+0.05},{1/(\x+5.9)+3});
      \draw[line width=0.8pt, domain=-5.9:-4.1,smooth,variable=\x,] plot ({\x},{-1/((\x+6.05)*(\x+3.95))-1.3});
      \draw[line width=0.8pt, domain=-3.9:-2.2,smooth,variable=\x] plot ({\x},{(\x+3)/((\x+4.11)*(\x+1.9))});
      \draw[line width=0.8pt, domain=-1.9:-1,smooth,variable=\x] plot ({\x-0.1},{(\x+3)/((\x+4.1)*(\x+2.05))-1.2});
      \draw[line width=0.8pt, domain=1:2.46,smooth,variable=\x] plot ({\x},{-(\x-3)/((\x-4.1)*(\x-2.55))-1.5});
      \draw[line width=0.8pt, domain=2.5:5,smooth,variable=\x] plot ({\x},{-0.65*(\x-3.5)/((\x-5.2)*(\x-2.35))+0.3});
      \draw[line width=0.8pt, domain=5.2:6.69,smooth,variable=\x] plot ({\x},{-0.5*(\x-6)/((\x-5.05)*(\x-6.8))+0.3});
      \node at (-4,-2.1) {$\Ind_a^b\Big(\displaystyle{\frac{Q}P}\Big) = 1 + 0 +1 = 2 $};
      \node at (4,-2.1) {$\Ind_a^b\Big(\displaystyle{\frac{Q}P}\Big) = -1 -1 -\frac12 = -\frac52$};
      \end{tikzpicture}
\end{center}

Note that with this extended definition of the Cauchy index, the Cauchy index of a rational function on an interval belongs to $\frac12 \Z$ and it is not necessarily an integer number.

\subsection{Sturm sequences and Cauchy index}

\begin{defn} Let $P, Q \in \R[X], P\not=0$. 
Define $S_0=P$ and,
if $Q\not=0$,
\begin{eqnarray*}
S_1&=&Q,\\
S_{2}&=&-
{\rm Rem}(S_{0},S_{1}),\\
&\vdots&\\
S_{i+1}&=&-
{\rm Rem}(S_{i-1},S_{i}),\\
&\vdots&\\
S_{s+1}&=&-
{\rm Rem}(S_{s-1},S_{s})=0,
\end{eqnarray*}
with $S_1, S_2, \dots , S_s \ne 0$, where ${\rm Rem}$ is the remainder 
in the euclidean division in $\R[X]$ of the first polynomial 
by the second polynomial. 

The Sturm sequence  of $P$ and $Q\not=0$
 is $(S_0,\ldots,S_s)$ and the Sturm sequence  of $P$ and $0$ is $S_0$, with $s=0$. We denote by $(d_0,\ldots,d_s)$ the degrees of $(S_0,\ldots,S_s)$.
\end{defn}

\begin{example}\label{example0}
Let $\alpha, \beta \in \R$ and $P = X^5 + \alpha X + \beta \in \R[X]$.
If  $\alpha \ne 0$ and $256\alpha^5 + 3125\beta^4 \ne 0$, 
the Sturm sequence of $P$ and $P'$ is $(S_0,S_1,S_2,S_3)$ with
$$
\begin{array}{lcl}
S_0 &=&  P = X^5 + \alpha X + \beta,  \\[2mm]
S_1 &=& P' = 5X^4 + \alpha,  \\[2mm]
S_2 &=&   \frac{-4\alpha}{5} X - \beta\\[2mm]
S_3 &=& \frac{-(256\alpha^5 + 3125\beta^4)}{256\alpha^4} .
\end{array}
$$
In this case,
$d_0 =5, d_1 = 4, d_2 = 1$, $d_3 = 0$.
\end{example}

Extending the classical results by Sturm (\cite{Stu})
and recent results by
\cite{Eis},  
we now explain that the Sturm sequence of $P$ and $Q$  
gives a formula for the general definition of the Cauchy index on an interval $[a, b]$ under 
no assumptions on $a$ and $b$. 
To do so, 
it is first needed to extend the notion of sign of a rational
function to degenerate cases.

\begin{defn} Let $P, Q \in \R[X] \setminus\{0\}$. 
Using the same notation as in Definition \ref{defn:CI_at_a_point}, we define
$$
\sign\Big(\frac{Q}P, x\Big) = \left\{
\begin{array}{ll}
\sign\big(
 \displaystyle{
 {\widetilde Q(x)}{\widetilde P(x)}
 }
 \big) \in \{-1, 1\} & \hbox{if }m=0, \\[3mm]
 0  & \hbox{otherwise } .
\end{array}
\right.
$$
We define also $\displaystyle{\sign\Big(\frac{Q}P, x\Big)} = 0$ if $Q=0$.
\end{defn}

In other words, if $x$ is a pole of
$\displaystyle{\frac{Q}P}$, the sign of $\displaystyle{\frac{Q}P}$ at $x$
is $0$; otherwise, it is simply the 
sign of the continuous extension of $\displaystyle{\frac{Q}P}$ at $x$.
Notice that 
if $Q\ne 0$, 
$\displaystyle{\sign\Big(\frac{Q}P, x\Big) = 
\sign\Big(\frac{P}Q, x\Big)}$.

We now state the general result relating the Cauchy index and the Sturm sequence, 
which will be 
proved at the end of Section \ref{sec:sigma_tau}.

\begin{theorem}\label{thm:eiser}
Let
$a, b \in \R$ with $a < b$ 
and $P, Q
\in \R[X], P\not=0$, 
$\deg Q = q < \deg P = p $.  
If $(S_0, \dots, S_s)$ is the Sturm sequence of $P$ and $Q$, then
$$
\Ind_a^b\Big(\frac{Q}{P}\Big) 
= 
\frac{1}{2}\sum_{0 \le i \le s-1} 
\left(
\sign\Big( \frac{S_{i+1}}{S_{i}}, b \Big) 
-
\sign\Big( \frac{S_{i+1}}{S_{i}}, a \Big) 
\right).
$$
\end{theorem}

Adding the condition that $a$ and $b$ 
are not common roots of $P$ and $Q$, from Theorem \ref{thm:eiser} a
\emph{sign-variation-counting} formula for the Cauchy index is obtained.

\begin{defn}
Let $x \in \R$ and $P, Q \in \R[X]$,
we 
define
 the sign variation of $(P, Q)$ at $x$ 
by
$${\rm Var}_x(P,Q)= \frac12  \Big|\sign(P(x)) - \sign(Q(x)) \Big|.$$
If $a, b \in \R$ with $a < b$, we denote
by ${\rm Var}_a^b(P, Q)$
the sign variation of $(P, Q)$ at $a$ minus
the sign variation of $(P, Q)$ at $b$;  namely, 
$${\rm Var}_a^b(P, Q)=
{\rm Var}_a(P, Q) -{\rm Var}_b(P, Q).
$$
\end{defn}

Note that for $x \in \R$, 
$$
{\rm Var}_x(P, Q)
= \left\{\begin{array}{cl}
          0 & \hbox{if } P(x) \hbox{ and } Q(x) \hbox{ have same sign},\\[1mm]
	            1 & \hbox{if } P(x) \hbox{ and } Q(x) \hbox{ have opposite non-zero sign},\\[1mm]
	  \frac12 & \hbox{if exactly one of } P(x) \hbox{ and } Q(x) \hbox{ has zero sign}.
	  \end{array}
\right.
$$
Moreover, if $x$ is not a common root of $P$ and $Q$, then
\begin{equation}
\label{signvar}
\sign\Big(\frac{Q}{P}, x\Big) = 
1 -  2{\rm Var}_x(P, Q).
\end{equation}

The following result then follows clearly from Theorem \ref{thm:eiser}.
\begin{theorem}\label{thm:eiser_variations}
Let
$a, b \in \R$ with $a < b$ 
and 
 $P, Q
\in \R[X], P\not=0$, 
$\deg Q = q < \deg P = p $. 
If
$a$ and $b$ are not common roots of $P$ and $Q$
and $(S_0, \dots, S_s)$ is the Sturm sequence of $P$ and $Q$, then
$$
\Ind_a^b\Big(\frac{Q}{P}\Big) 
= 
\sum_{0 \le i \le s-1} 
{\rm Var}_a^b(S_{i}, S_{i+1}).$$
\end{theorem}

Theorem \ref{thm:eiser_variations} is a generalization of the classical Sturm theorem 
\cite{Stu,BPRbook},
since $a$ or $b$ can be root of $P$ or $Q$ (but not of both).

\subsection{Subresultant polynomials}

Subresultant polynomials are polynomials which are proportional to the ones in the 
 Sturm sequence, but enjoy better properties since their coefficients belong to the ring generated by the coefficients of $P$ and $Q$.
We include  definitions and properties concerning subresultant polynomials. We refer
the reader to \cite{BPRbook} for proofs and details.

Let $\D$ be a domain and let ${\rm ff}(\D)$ be its fraction field.

\begin{defn}\label{def:subres} Let $P = a_pX^p + \dots + a_0, 
Q = b_qX^q + \dots + b_0 \in \D[X] \setminus \{0\}$ with $\deg P
 = p \ge 1$ 
and
$\deg Q
 = q < p$. 
For $0 \le j \le q$, the $j$-th 
subresultant polynomial of 
$P$ and $Q$, ${\sResP}_j (P, Q) \in \D[X]$ 
is 
$$
\det
\underbrace{
\left(\begin{array}{cccccccc}
 a_p & a_{p-1}  & \dots   &          &            &            &        & X^{q-j-1}P \cr
 0   & a_p      & a_{p-1} &  \dots   &            &            &        & X^{q-j-2}P \cr
     &  \ddots  & \ddots  &  \ddots  &            &            &        & \vdots \cr
     &          & 0       & a_p      &  a_{p-1}   & \dots      &        & P \cr
     &          &         & 0        & b_q        &  b_{q-1}   &\dots   & Q \cr
     &          & \iddots & \iddots  &  \iddots   &            &        & \vdots \cr
     &  \iddots & \iddots &  \iddots &            &            &        & \vdots \cr
 0   & b_q      & b_{q-1} & \dots    &            &            &        & X^{p-j-2}Q \cr
 b_q & b_{q-1}  & \dots   &          &            &            &        &X^{p-j-1}Q \cr
  \end{array}\right)
}
 \in \D[X].
$$
$$
p+q-2j
$$ 
By convention, we
  extend this definition 
  with
  \begin{eqnarray*}
    {\sResP}_p (P, Q) & = & P \ \in \D[X],\\
    {\sResP}_{p - 1} (P, Q) & = & Q \ \in \D[X],\\
    {\sResP}_j (P, Q) & = & \ 0  \ \in \D[X] \qquad \hbox{ for } q < j < p - 1.
  \end{eqnarray*}
We also define ${\sResP}_p (P, 0)=P, {\sResP}_j (P, 0)=0, j=0,\ldots,p-1.$
\end{defn}

Note that in the matrix above, all the entries in the first $p + q - 2j -1$ columns are elements in $\D$, 
and all the entries in the last column are elements in $\D[X]$. 
Doing column operations, it is easy to prove that for $0 \le j \le p$, 
$$\deg \sResP_j(P, Q) \le j.$$ 
Note also that in the case $q = p-1$, we have given two definitions for $\sResP_q(P, Q)$,
both equal to $Q$ so that 
there is no ambiguity.

\begin{defn} \label{def:subres2}
Let $P, Q \in \D[X] \setminus \{0\}$ with $\deg P
 = p \ge 1$ 
and
$\deg Q
 = q < p$. 

\begin{itemize}  
  
\item For $0 \le j \le q$, 
the $j$-th 
subresultant coefficient of $P$ and $Q$, ${\sRes}_j (P,
Q)\in \D$  is the coefficient of $X^j$ in
${\sResP}_j (P, Q)$. 
By convention, we extend this definition with  
\begin{eqnarray*}
    {\sRes}_p(P, Q) & = & 1 \ \in \D \qquad \hbox{(even if } P \hbox{ is not monic)},\\
    {\sRes}_j(P, Q) & = & 0 \ \in \D \qquad \hbox{ for } q < j \le p - 1.
  \end{eqnarray*}

\item For $0 \le j \le p$, ${\sResP}_j (P, Q)$ is said to be 
\begin{itemize}
\item
{\bf defective} 
if $\deg {\sResP}_j (P, Q) < j$ or, equivalently, if ${\sRes}_j (P, Q)
= 0$, 
\item
{\bf non-defective} 
if $\deg {\sResP}_j (P, Q) =j$ or, equivalently, if ${\sRes}_j (P, Q)
\not= 0$.
\end{itemize}
\end{itemize}

\end{defn}

We refer the reader to \cite[Chapitre 9]{Ape_Jou_book} for another
definition of subresultant polynomials and coefficients, which differs possibly in a sign.

We illustrate Definitions \ref{def:subres} and \ref{def:subres2}
with the following example. 

\begin{example}\label{ex:example}
Let $\alpha, \beta \in \R$ and $P = X^5 + \alpha X + \beta \in \R[X]$, then $P' = 5X^4 + \alpha$ as in Example \ref {example0}. We have $\sResP_5(P, P')= P, \sResP_4(P, P')= P'$
$$
\begin{array}{lcccl}
\sResP_3(P, P') &=& \det 
\left(
\begin{array}{ccc}
1 & 0 & X^5 + \alpha X + \beta \\
0 & 5 & 5X^4 + \alpha \\
5 & 0 & 5X^5 + \alpha X
\end{array} 
\right)
&=&
-5(4\alpha X + 5\beta)\\[8mm]
\sResP_2(P, P')& = &\det 
\left(
\begin{array}{ccccc}
1 & 0 & 0       & 0       & X^6 + \alpha X^2 + \beta X^1 \\
0 & 1 & 0       & 0       &  X^5 + \alpha X + \beta \\
0 & 0 & 5       & 0       & 5X^4 + \alpha \\
0 & 5 & 0       & 0       & 5X^5 + \alpha X^1 \\
5 & 0 & 0       & 0       &  5X^6 + \alpha X^2 \\
\end{array} 
\right)
&=& 0,\\[14mm]
\sResP_1(P, P')& = &\det 
\left(
\begin{array}{ccccccc}
1 & 0 & 0 & 0 & \alpha & \beta   & X^7 + \alpha X^3 + \beta X^2 \\
0 & 1 & 0 & 0 & 0      & \alpha  & X^6 + \alpha X^2 + \beta X^1 \\
0 & 0 & 1 & 0 & 0      & 0       &  X^5 + \alpha X + \beta \\
0 & 0 & 0 & 5 & 0      & 0       & 5X^4 + \alpha \\
0 & 0 & 5 & 0 & 0      & 0       & 5X^5 + \alpha X^1 \\
0 & 5 & 0 & 0 & 0      & \alpha  &  5X^6 + \alpha X^2 \\
5 & 0 & 0 & 0 & \alpha & 0       &  5X^7 + \alpha X^3 \\
\end{array} 
\right)
&= &
80 \alpha^2(4\alpha X + 5\beta), 
\\[20mm]
\sResP_0(P, P') &= &\det 
\left(
\begin{array}{ccccccccc}
1 & 0 & 0 & 0 & \alpha & \beta & 0 & 0 & X^8 + \alpha X^4 + \beta X^3 \\
0 & 1 & 0 & 0 & 0 & \alpha & \beta & 0 & X^7 + \alpha X^3 + \beta X^2 \\
0 & 0 & 1 & 0 & 0 & 0 & \alpha & \beta & X^6 + \alpha X^2 + \beta X^1 \\
0 & 0 & 0 & 1 & 0 & 0 & 0 & \alpha &  X^5 + \alpha X + \beta \\
0 & 0 & 0 & 0 & 5 & 0 & 0 & 0 & 5X^4 + \alpha \\
0 & 0 & 0 & 5 & 0 & 0 & 0 & \alpha & 5X^5 + \alpha X^1 \\
0 & 0 & 5 & 0 & 0 & 0 & \alpha & 0 & 5X^6 + \alpha X^2 \\
0 & 5 & 0 & 0 & 0 & \alpha & 0 & 0 & 5X^7 + \alpha X^3 \\
5 & 0 & 0 & 0 & \alpha & 0 & 0 & 0 & 5X^8 + \alpha X^4 \\
\end{array} 
\right)
&= &256\alpha^5 + 3125\beta^4.
\end{array}
$$

Note that $\sResP_5(P, P')$ and $\sResP_4(P, P')$ are non-defective 
while
$\sResP_3(P, P')$ and 
$\sResP_2(P, P')$ are defective. Finally,
$\sResP_1(P, P')$ is defective if and only if $\alpha = 0$, and 
$\sResP_0(P, P')$ is defective if and only if $256\alpha^5 + 3125\beta^4= 0$.
\end{example}

The following Structure Theorem is 
a key result in the theory 
of subresultants, stating 
 the connection between subresultants and 
remainders.
To state it, we  need to introduce a notation. 

\begin{notn}\label{notn:epsilon}
 For $n \in \Z$, we denote
$\displaystyle{
 \epsilon_{n} = (-1)^{\frac12n(n-1)}.
}$
\end{notn}

Note that $\epsilon_n = 1$ if the remainder of $n$ in the division by $4$ is $0$ or $1$ and
$\epsilon_n = -1$ if the remainder of $n$ in the division by $4$ is $2$ or $3$; 
this implies that for $k \in \Z$
\begin{equation}\label{rem:epsilon_bis}
\epsilon_{2k+n} = (-1)^k\epsilon_{n} =\epsilon_{2k}\epsilon_{n}.
\end{equation}

\begin{theorem}[Structure Theorem of Subresultants]
\label{thm:structure_thm_subresultants} 
Let $P, Q
\in \D[X] \setminus \{0\}$ 
with $\deg P
 = p \ge 1$ and $\deg Q
  = q < p$.
Let 
$(d_0,\ldots,d_s)$ be
the sequence of degrees of the 
Sturm sequence of $P$ and 
$Q$ in decreasing order
and let 
$d_{-1} = p+1$ (note that $d_0 = p$ and $d_1 = q$). 
\begin{itemize}
\item For $1 \le i \le s$, 
$$
{\sResP}_{d_{i-1}-2}(P,Q) = \dots = 
{\sResP}_{d_{i} + 1}(P,Q) = 0 \in \D[X]
$$
and 
${\sResP}_{d_{i-1}-1}(P,Q)$   and
${\sResP}_{d_{i}}(P,Q)$  
 are proportional. More precisely,
for $1 \le i \le s$, denote
$$
\begin{array}{rcl}
T_i & = & {\sResP}_{d_{i-1}-1}(P,Q) \ \in \D[X],\\
    t_i & = & {\rm lc}(T_i) \ \in \D\\
\end{array}
$$
(note that $T_1 = Q$), and extend this notation with $T_0 = P$ and $t_{0}= 1 \in \D$
(even if $P$ is not monic).
Then
$$
{\sRes}_{d_{i}}(P,Q) \cdot T_i =  
t_i \cdot {\sResP}_{d_{i}}(P,Q) \in \D[X]
$$
with
\begin{equation}\label{eq:constant_STS_1}
{\sRes}_{d_{i}}(P,Q)
=
\epsilon_{d_{i-1}-d_i}\cdot
\frac{t_i^{d_{i-1}-d_i}}
{{\sRes}_{d_{i-1}}(P,Q)^{d_{i-1}-d_i-1}} \in \D.
\end{equation}
This implies $\deg T_i = d_i \le d_{i-1} - 1$.

\item For $1 \le i \le s-1$, 
\begin{equation}\label{eq:constant_STS_2}
t_{i-1} \cdot 
{\sRes}_{d_{i-1}}(P,Q)
\cdot T_{i+1}
=
-
{\rm Rem}
\left(
t_i \cdot 
{\sRes}_{d_{i}}(P,Q)
\cdot T_{i-1}
,
T_i 
\right) 
\in \D[X]
\end{equation}
(where ${\rm Rem}$ is the remainder 
in the euclidean division in ${\rm ff}(\D)[X]$ of the first polynomial 
by the second polynomial) and the quotient belongs to $\D[X]$.

\item Both $T_s \in \D[X]$ 
and ${\sResP}_{d_{s}}(P, Q) \in \D[X]$
are 
greatest common divisors of $P$ and $Q$ in ${\rm ff}(\D)[X]$ 
and they
divide  ${\sResP}_{j}(P, Q)$ for $0 \le j \le p$. 
In addition, if $d_s > 0$ then
$$
{\sResP}_{d_{s}-1}(P, Q) = \dots = {\sResP}_{0}(P, Q) = 0 \in \D[X].
$$
\end{itemize}
\end{theorem}

\begin{proof}{Proof:} See
\cite[Chapter 8]{BPRbook}.
\end{proof}

Note that Theorem \ref{thm:structure_thm_subresultants} (Structure Theorem of Subresultants) gives a method for computing the subresultant polynomials using remainders which is more efficient than using their definition as determinants. However we are not concerned with subresultant polynomials computations in the current paper. We are only concerned with a formula for the Cauchy index using the subresultant polynomials.

Theorem \ref{thm:structure_thm_subresultants}  (Structure Theorem of Subresultants) can be illustrated by the following picture.

{\small
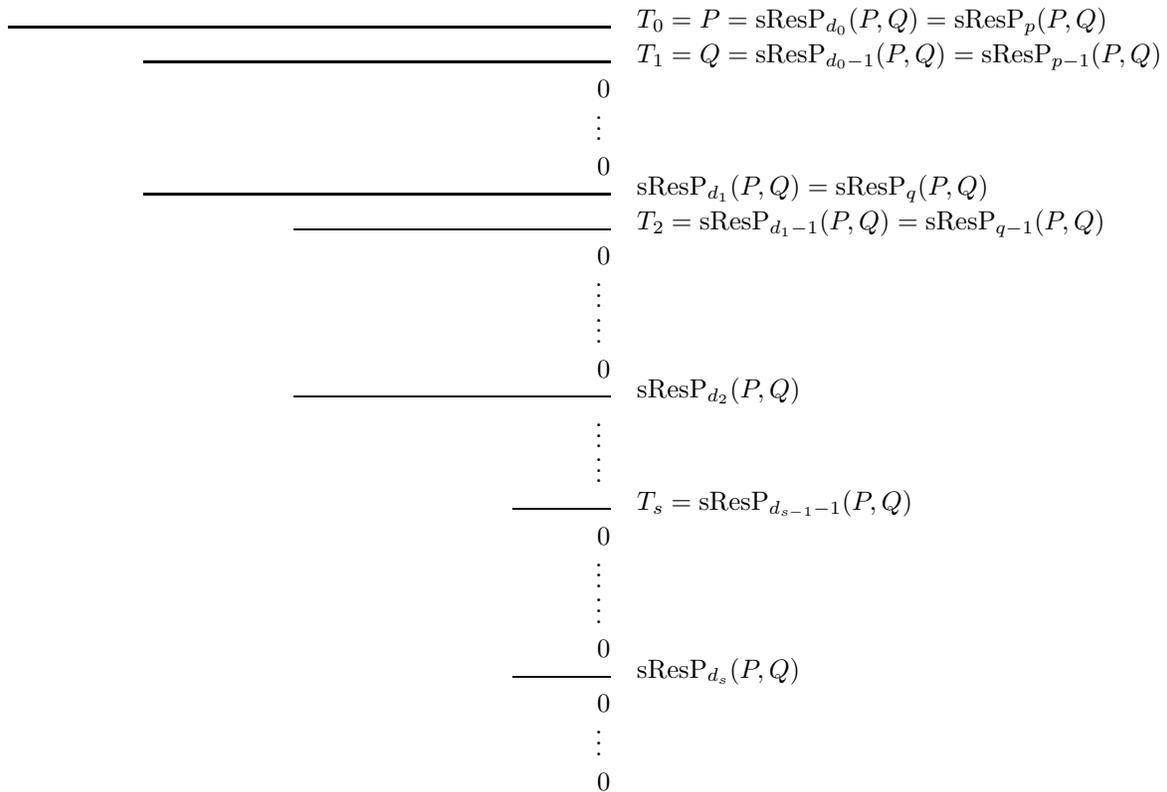
$$\begin{array}{rl}
\rule{8cm}{0.4pt} & T_0=P=\sResP_{d_0}(P, Q) = \sResP_p(P, Q)\\
\rule{6.2cm}{0.4pt} &T_1=Q= \sResP_{d_0 - 1}(P, Q) = \sResP_{p-1}(P, Q)\\
0 & \\[-1mm]
\vdots  {\hspace{0.05cm} }& \\
0 & \\[-2mm]
\rule{6.2cm}{0.4pt} & \sResP_{d_1}(P, Q) = \sResP_{q}(P, Q) \\
\rule{4.2cm}{0.4pt}& T_2=\sResP_{d_1 - 1}(P, Q) = \sResP_{q-1}(P, Q)  \\
0 & \\ [-1mm]
\vdots {\hspace{0.05cm} } & \\[-2mm]
\vdots {\hspace{0.05cm} } & \\
0 & \\ [-2mm]
\rule{4.2cm}{0.4pt} & \sResP_{d_2}(P, Q) \\
\vdots {\hspace{0.05cm} } & \\[-2mm]
\vdots {\hspace{0.05cm} } & \\[-1mm]
\rule{1.3cm}{0.4pt}& T_s=\sResP_{d_{s-1} - 1}(P, Q)  \\
0 & \\ [-1mm]
\vdots {\hspace{0.05cm} } & \\[-2mm]
\vdots {\hspace{0.05cm} } & \\
0 & \\ [-2mm]
\rule{1.3cm}{0.4pt} & \sResP_{d_s}(P, Q)  \\
0 & \\ [-1mm]
\vdots {\hspace{0.05cm} } & \\
0 & \\ [-2mm]
\end{array}
$$
}

As a corollary to Theorem \ref{thm:structure_thm_subresultants},
all subresultant polynomials are either $0$ or proportional to polynomials in the Sturm sequence. More precisely, for $1 \le i \le s$,
the subresultant polynomial $T_i$ is proportional to $S_i$ in the Sturm sequence.

\begin{remark}\label{rem:non-defective}
In the case where all the subsresultant polynomials are 
non-defective, 
there are no pairs of proportional polynomials in the sequence of subresultant polynomials,
the degrees of the polynomials $S_i$ in the Sturm sequence decrease one by one, 
and the coefficient of proportionality between $T_i$ and $S_i$  is a square (see \cite[Corollary 8.37]{BPRbook}).
\end{remark}

\begin{example}
[Continuation of   Example \ref{example0} and Example \ref{ex:example}] \label{ex:example2}

Let us take as before $P=X^5+\alpha X+\beta$ and suppose $\alpha \ne 0$ and $256\alpha^5 + 3125\beta^4 \ne 0$.

Looking at 
Example \ref{ex:example}, we observe that, as expected given
the Structure Theorem,  the degrees of the non-defective subresultant polynomials are 
$d_0 =5, d_1 = 4, d_2 = 1$, $d_3 = 0$, i.e. the degrees of the polynomials in the Sturm sequence given in  Example \ref{example0}.
Moreover
$\sResP_2(P, P') = 0$, while $\sResP_3(P, P')$ is proportional to $\sResP_1(P, P')$ 
and, 
also, to $S_2$ given in  Example \ref{example0}.

Using the notation from Theorem \ref{thm:structure_thm_subresultants}, we have
$$
\begin{array}{lcl}
T_0 = \sResP_5(P, P') = P = X^5 + \alpha X + \beta, & & t_0 = 1, \\[1mm]
T_1 = \sResP_4(P, P') = P' = 5X^4 + \alpha, & & t_1 = 5, \\[1mm]
T_2 = \sResP_3(P, P') = -20\alpha X - 25 \beta, & & t_2 = -20\alpha, \\[1mm]
T_3 = \sResP_0(P, P') = 256\alpha^5 + 3125\beta^4, & & t_3 = 256\alpha^5 + 3125\beta^4. 
\end{array}
$$

\end{example}

\subsection{Main results}

In order to state our 
results we introduce the following notation.

\begin{notn}\label{notn:lead_non_def} 
Using the notation from Theorem \ref{thm:structure_thm_subresultants},
for  $0 \le i \le s$, let 
$$
{\rm p}(i) = \max\{j \ | \ 0 \le j \le i, \ d_{j-1} - d_{j} \hbox{ is odd}\}
$$
(${\rm p}(i)$ is  well-defined since $d_{-1} - d_0 = (p+1) - p = 1$ is odd). 
\end{notn}

We are ready now to state our main result, which is  a new formula 
for
${\rm Ind}_a^b\Big(\displaystyle{\frac{Q}{P}}\Big)$
using only the polynomials $T_i$ in the sequence of the subresultant polynomials.

\begin{theorem}\label{thm:main_thm}
Let
$a, b \in \R$ with $a < b$ 
and $P, Q
\in \R[X] \setminus \{0\}$ 
with $\deg P
 = p \ge 1$ and $\deg Q
  = q < p$. Then
$$
\Ind_a^b\Big(\frac{Q}{P}\Big) 
= 
\frac{1}{2}\sum_{0 \le i \le s-1} 
\epsilon_{d_{{\rm p}(i) - 1} - d_i}\cdot {\rm sign}(t_{{\rm p}(i)}) \cdot {\rm sign}(t_i) \cdot
\left(
\sign\Big( \frac{T_{i+1}}{T_{i}}, b \Big) 
-
\sign\Big( \frac{T_{i+1}}{T_{i}}, a \Big) 
\right).
$$
\end{theorem}

As we will see in Section \ref{sec:comcon}, the main advantage of the formula in 
Theorem \ref{thm:main_thm} in comparison with previously known related formulas is that 
there is no assumption on the endpoints $a$ and $b$ of the interval, and, more importantly, 
potentially 
less subresultant polynomials (i.e. only te $T_i$) are involved.

If we add the condition that $a$ and $b$ 
are no roots of $P$ and $Q$, from Theorem \ref{thm:main_thm} we obtain a
\emph{sign-variation-counting} formula.

\begin{theorem}\label{thm:main_thm_variations}
Let
$a, b \in \R$ with $a < b$ 
and $P, Q
\in \R[X] \setminus \{0\}$ 
with $\deg P
 = p \ge 1$ and $\deg Q
  = q < p$.
If
$a$ and $b$ are not common roots of $P$ and $Q$, then
$$
\Ind_a^b\Big(\frac{Q}{P}\Big) 
= 
\sum_{0 \le i \le s-1} 
\epsilon_{d_{{\rm p}(i) - 1} - d_i}\cdot {\rm sign}(t_{{\rm p}(i)}) \cdot {\rm sign}(t_i) \cdot
{\rm Var}_a^b(T_i, T_{i+1}).$$
\end{theorem}

Finally, for the Cauchy index on $\R$, we obtain the following result. 

\begin{theorem}\label{thm:ind_all_R}
Let
$P, Q
\in \R[X] \setminus \{0\}$ 
with $\deg P
 = p \ge 1$ and $\deg Q
  = q < p$.
If the leading coefficient of $P$ is positive or if $d_0 -d_1 = p-q$ is even, then 
$$
{\rm Ind}_{\R}\Big( \frac{Q}{P} \Big) =
\sum_{0 \le i \le s-1, \atop  d_{i} - d_{i+1} {\scriptsize\hbox{odd}}} 
\epsilon_{d_{{\rm p}(i) - 1} - d_i}\cdot {\rm sign}(t_{{\rm p}(i)}) 
\cdot {\rm sign}(t_{i+1}). 
$$
If the leading coefficient of $P$ is negative and $d_0 - d_1 = p-q$ is odd, then 
$$
{\rm Ind}_{\R}\Big( \frac{Q}{P} \Big) =
- {\rm sign}(t_{1})
+ 
\sum_{1 \le i \le s-1, \atop  d_{i} - d_{i+1} {\scriptsize\hbox{odd}}} 
\epsilon_{d_{{\rm p}(i) - 1} - d_i}\cdot {\rm sign}(t_{{\rm p}(i)}) 
\cdot {\rm sign}(t_{i+1}). 
$$
\end{theorem}

\begin{example}[Continuation of Examples  \ref{example0}, \ref{ex:example} and \ref{ex:example2}] \label{ex:example3}
Following Notation \ref{notn:lead_non_def}, for $0 \le i \le 3$ we have ${\rm p}(i) = i$. Therefore, 
by Theorem \ref{thm:ind_all_R},
when we fix $(\alpha,\beta)\in \R^2$ with $\alpha \ne 0$ and $256\alpha^3 + 3125\beta^4 \ne 0$, the number of roots of $P$ in $\R$ is given by
$$
{\rm Ind}_{\R}\Big( \frac{P'}{P} \Big) = 1 - {\sign}(\alpha) + {\sign}(\alpha)\cdot {\sign}(256\alpha^5 + 3125\beta^4).
$$
\end{example}

The rest of the paper is organized as follows. 
In Section \ref{sec:comcon} we 
comment the differences between our results
and previously known related formulas.
In Section \ref{sec:properties_Ci} we review some useful properties of Cauchy index. 
In Section \ref{sec:sigma_tau}, we recall the notion of $(\sigma, \tau)$-chain and 
their connection with Cauchy index. 
Finally, in Section \ref{sec:proofmain} we prove Theorem \ref{thm:main_thm} 
using $(\sigma,\tau)$-chains, and Theorems \ref{thm:main_thm_variations} and \ref{thm:ind_all_R}
as consequences of Theorem \ref{thm:main_thm}.

\section{Comparison with previous Cauchy index 
formulas using subresultants}\label{sec:comcon}

There is a previously known formula 
for
 the Cauchy index 
$\displaystyle{{\rm Ind}_a^b\Big( \frac{Q}{P}\Big)}$
by means of subresultant 
polynomials which is as follows (see \cite[Chapter 9]{BPRbook}).

\begin{defn}
{\em
Let $s$ be
a finite sequence
of  $n$ elements in $\R$
of type 
$$
s=(s_n, \underbrace{0 , \ldots, 0}_{n-m-1}, \underbrace{ \textcolor{white}{0, \ldots} s' \textcolor{white}{0, \ldots} }_{m \hbox{\small{ elements}}},
)$$ 
with $s_n \not=0$ and
$s'$ a finite sequence
of $m$ elements in $\R$ with $0 \le m \le n-1$,
which is either empty (this is, $m= 0$) or 
$s'=(s_m,\ldots,s_1)$ with $s_m\not=0$.
The \textbf{modified number of sign variations}
\index{Number of sign variations} in $s$ is
defined inductively as follows
$$
 \W(s)=
\begin{cases}
 0 & \mbox{if } s'=\emptyset, \\
\W(s')+1&\mbox{if } s_ns_m<0,\\
\W (s')+2&\mbox{if }  s_ns_m>0 \mbox{ and }  n-m=3,\\
\W (s')&\mbox{if }   s_ns_m>0 \mbox{ and }  n-m\not=3.
\end{cases}
$$
In other words,
the usual definition of the number
of sign variations is modified by counting two sign variations
 for the groups: $+,0,0,+$ and
$-,0,0,-$.
If there are no zeros in the sequence $s$,
$\W(s)$
is just the classical number of sign variations in the sequence.

Let ${\cal P}$ be a sequence $(P_0,P_1,\ldots,P_d)$ 
of
polynomials in 
$\R[X]$
and let $x$ be an element of
$\R
$ which is not
a root of the $\gcd$ of $(P_0, P_1, \dots, P_d)$, which we call $\gcd({\cal P})$.
Then $\W({\cal P};x)$, the
 \textbf{modified number of sign variations}
of
${\cal P}$ at $x$, is the number defined as
 follows:
\begin{itemize}
\item[-] delete from ${\cal P}$ those polynomials
that are identically $0$
to obtain the sequence of polynomials
 ${\cal Q}=(Q_0,\cdots,Q_s)$ in $\D[X]$,
\item[-] define $\W({\cal P};x)$ as
$\W(Q_0(x),\cdots,Q_s(x))$.
\end{itemize}
Let  $a$ and $b$ be elements of
$\R$ which are not
roots of $\gcd({\cal P})$.
The difference
between the number of modified sign variations in
${\cal P}$ at $a$ and $b$
is denoted by
\[\W({\cal P};a,b)=\W({\cal P};a)-\W({\cal P};b).\]
}
\end{defn}

Denoting by $\SR(P,Q)$ the list of subresultant polynomials of $P$ and $Q$, 
the
following result is known (see \cite[Chapter 9]{BPRbook}).

\begin{proposition}
\label{9:the:sthatheoremn}
Let $a, b \in \R$ with $a < b$ and $P, Q \in \R[X] \setminus \{0\}$ with $\deg P = p \ge 1$
and $\deg Q = q < p$. If $a$ and $b$ are not roots of $P$, then
$$
\Ind_a^b \Big(\displaystyle{\frac{Q}P}\Big) = \W(\SR(P,Q);a,b).
$$
 \end{proposition}

Our new formula 
for
${\rm Ind}_a^b \Big(\displaystyle{\frac{Q}P}\Big)$ given in Theorem \ref{thm:main_thm} 
improves on the one from Proposition \ref{9:the:sthatheoremn} in several aspects: 
 
\begin{itemize}
\item[a.i)] Theorem \ref{thm:main_thm} is general,
 there are no restrictions on   $a$ and $b$.
 
\item[a.ii)] 
More importantly, there are cases when
less subresultant polynomials are involved in this new formula. 
The Structure Theorem
of Subresultants (Theorem \ref{thm:structure_thm_subresultants} ) states that in the subresultant polynomial
sequence, some polynomials appear only once and other polynomials appear exactly twice  (up to scalar multiples). 
In addition, if a polynomial appears twice, its first appearance, $T_i$, is defined as the polynomial determinant of  
a matrix of smaller size (in comparison with its second appearance), so that it is more 
suitable in computations.
Our formula involves only the $T_i$, i.e. the first appearance( up to scalar multiples) of each polynomial
in the subresultant
polynomial sequence.
\end{itemize}

In the special case when $a$ and $b$ are not common roots of $P$ and $Q$, 
Theorem \ref{thm:main_thm_variations} gives a \emph{sign-variation-counting} formula which
improves on the one from Proposition \ref{9:the:sthatheoremn}
since:
\begin{itemize}
\item[b.i)] Theorem \ref{thm:main_thm_variations} imposes less restrictions on $a$ and $b$. 
\item[b.ii)] 
As in  a.ii). 
\item[b.iii)] 
The formula
is more natural, since the sign-variation counting in Theorem \ref{thm:main_thm_variations} is local and needs only to consider the sign of two consecutive elements, contrarily to the modified number of sign variations
 which is very counter-intuitive. 
\end{itemize}

Last but not least, the proofs of our results are  also less technically involved 
than the proof of Proposition \ref{9:the:sthatheoremn},
which is cumbersome
(see the proof in \cite[Chapter 9]{BPRbook}).

Note that, in the particular case where all subresultant polynomials are non-defective, 
both the formulas in Theorem \ref{thm:main_thm_variations} and in Proposition 
\ref{9:the:sthatheoremn} become
$$
\Ind_a^b\Big(\frac{Q}{P}\Big) 
= 
\sum_{0 \le j \le p-1} 
{\rm Var}_a^b({\sResP}_{j}, {\sResP}_{j+1})$$
(see \cite[Chapters 2 and 9]{BPRbook}), 
but the new formula extends the previous one to the case 
whera $a$ and $b$ are not common roots of $P$ and $Q$.

There is also a previously known formula 
for
the Cauchy index ${\rm Ind}_{\R}\Big( \displaystyle{\frac{Q}{P}}\Big)$ by means of 
subresultant coefficients
which we introduce below
(see \cite[Chapter 4]{BPRbook}).

\begin{proposition}\label{oldnondef} Using the notation from Theorem \ref{thm:structure_thm_subresultants},
for $0 \le i \le s$, let $s_i={\sRes}_{d_i} (P, Q)$ be the leading coefficient of 
the non-defective subresultant polynomial ${\sResP}_{d_i} (P, Q)$ (which is 
proportional to $T_i$). Then
 $$
{\rm Ind}_{\R}\Big( \frac{Q}{P} \Big) =
\sum_{0 \le i \le s-1, \atop  d_{i} - d_{i+1} {\scriptsize\hbox{odd}}} 
\epsilon_{d_{i} - d_{i+1}}\cdot {\rm sign}(s_{i}) 
\cdot {\rm sign}(s_{i+1}).
$$
\end{proposition}

Even in this special case, the new formula 
for
${\rm Ind}_a^b\Big( \displaystyle{\frac{Q}{P}}\Big)$ given in 
Theorem \ref{thm:ind_all_R}
improves on the one from Proposition \ref{oldnondef}. 
As before, the main difference between the two formulas is that  the $t_i$ are, 
in the defective cases, defined as determinants of matrices of smaller sizes
than the $s_i$ and therefore 
 is more 
suitable in computations.
On the other hand, one advantage of Proposition \ref{oldnondef} is that it can be proved directly, using  only 
subresultant coefficients
and does not use the  definition of the subresultant polynomials and the Structure Theorem of subresultants (see \cite[Chapter 4]{BPRbook}).

\section{Properties of Cauchy index}\label{sec:properties_Ci}

In this section we include some useful properties of Cauchy index. 

\begin{lemma}\label{lem:mult_constant}
Let $a, b \in \R$ with
$a < b$, $P, Q \in \R[X] \setminus \{0\}$ and $c \in \R \setminus \{0\}$. 
Then 
$$\Ind_a^b\Big(\frac{c \cdot Q}{P}\Big)
= \sign(c) \cdot \Ind_a^b\Big(\frac{Q}{P}\Big).
$$
\end{lemma}
\begin{proof}{Proof:} Follows immediately from the definition of Cauchy index.  
\end{proof}

\begin{lemma}\label{lem:reduce}
Let $a, b \in \R$ with
$a < b$, $P, Q, R \in \R[X] \setminus \{0\}$ and $T \in \R[X]$
such that
$$
Q = P T + R.
$$
Then 
$$\Ind_a^b\Big(\frac{Q}{P}\Big)
= \Ind_a^b\Big(\frac{R}{P}\Big).
$$
\end{lemma}

\begin{proof}{Proof:} For each $x \in [a, b]$, we first note that if
$$
\frac{Q}{P} = (X - x)^{m}\frac{\widetilde{Q}}{\widetilde{P}}
$$
with $m\in \Z$, $\widetilde{P}(x) \ne 0, \widetilde{Q}(x) \ne 0$  and $m<0$,
then defining $$\widetilde{R}=\widetilde{Q}-(X-x)^{-m} \widetilde{P}T,$$
we have
$$
\frac{R}{P} = (X - x)^{m}\frac{\widetilde{R}}{\widetilde{P}}
$$
with $ \widetilde{P}(x) \ne 0$ and $\widetilde{R}(x)=\widetilde{Q}(x) \ne 0$.
This proves that ${\rm Ind}_x^\varepsilon \Big( \displaystyle{\frac{Q}{P} }\Big) =
{\rm Ind}_x^\varepsilon \Big( \displaystyle{\frac{R}{P} }\Big)$
for every $\varepsilon\in \{-1,1\}$. 
The claim follows from the definition of the Cauchy index.
\end{proof}

The following property is known as \emph{the inversion formula}.

\begin{proposition}\label{prop:inv_formula}
Let $a, b \in \R$ with
$a < b$ and $P, Q \in \R[X] \setminus \{0\}$. 
Then 
$$
\Ind_a^b\Big(\frac{Q}{P}\Big) + 
\Ind_a^b\Big(\frac{P}{Q}\Big)
=  \frac12\sign\Big(\frac{Q}P, b\Big) - \frac12 \sign\Big(\frac{Q}P, a\Big).
$$
\end{proposition}

\begin{proof}{Proof:} See \cite[Theorem 3.9]{Eis}. 
\end{proof}

\section{$(\sigma,\tau)$-chains and Cauchy index}\label{sec:sigma_tau}

The notion of $(\sigma, \tau)$-chain was introduced in \cite{PerRoy}.
Here, we need to introduce 
a slight variation of this notion.

\begin{defn}\label{defn:special_chain} Let 
$n\in\Z_{\ge 1}$ and
$\sigma, \tau \in \{-1, 1\}^{n-1}$
with $\sigma = (\sigma_1, \dots, \sigma_{n-1})$ and 
$\tau = (\tau_1, \dots, \tau_{n-1})$.
A sequence of polynomials $(P_0, \dots, P_n)$ in $\R[X]$
is a \emph{special} $(\sigma, \tau)$-\emph{chain} 
if for $1 \le i \le n-1$ there exist 
$a_i, c_i \in \R\setminus\{0\}$
and $B_i \in \R[X]$
such that
\begin{enumerate}
\item \label{it:1}
$
a_iP_{i+1} + B_iP_i + c_iP_{i-1} = 0,
$
\item \label{it:2} $\sign(a_i) = \sigma_i$,
\item $\sign(c_i) = \tau_i$. 
\end{enumerate}
\end{defn} 

As in \cite{PerRoy}, note that for $n = 1$, taking $\{-1, 1\}^0 = \{ \bullet \}$, 
any sequence $(P_0, P_1)$ in $\R[X]$ is a special $(\bullet, \bullet)$-chain. 

Note also that Sturm sequences are always special $(1,\ldots,1),(1,\ldots,1)$ chains.

\begin{example}[Continuation of Examples \ref{example0}, \ref{ex:example}, \ref{ex:example2} and \ref{ex:example3}]
Taking $\sigma = (1, 1)$ and $\tau = (1, 1)$, then
$(S_0, S_1, S_2, S_3)$ is a special $(\sigma, \tau)$-chain,
with
$$
\begin{array}{rcl}
a_1 &= &1, \\[2mm] 
B_1 & =&  -\frac{X}{5}, \\[2mm]
c_1 & = &1, \\[2mm] 
a_2 & =& 1, \\[2mm]
B_2 &=& 
\frac{25 (64 X^3 \alpha^3 - 80 X^2 \alpha^2 \beta + 100 X \alpha \beta^2 - 125\beta^3)}{256 \alpha^4},
\\[2mm] 
c_2 &=&1.
\end{array}
$$

Taking now $\sigma = (1, 1)$ and $\tau = (1, -1)$, then 
$(T_0, T_1, T_2, T_3)$ is a special $(\sigma, \tau)$-chain,
with
$$
\begin{array}{rcl}
a_1 &= &1, \\[2mm] 
B_1 & =&  -5X, \\[2mm]
c_1 & = & 25, \\[2mm] 
a_2 & =& 25, \\[2mm]
B_2 &=& 
-25(64 X^3\alpha^3 - 80 X^2 \alpha^2 \beta + 100 X \alpha \beta^2 - 125 \beta^3),
 \\[2mm] 
c_2 &=& -6400\alpha^4. 
\end{array}
$$
We will see in Section \ref{sec:proofmain} how to produce
special $(\sigma, \tau)$-chains using Theorem \ref{thm:structure_thm_subresultants} (Structure Theorem of Subresultants). 
\end{example}

We introduce some more useful definition.

\begin{defn}
Let $a, b \in \R$, 
$n \in \Z_{\ge 1}$, 
$(P_0, \dots, P_n)$ in $\R[X]\setminus \{ 0\}$
and $\sigma, \tau \in \{-1, 1\}^{n-1}$. We define
 $\theta(\sigma, \tau)_0 = 1$,
for $ 1 \le i \le n-1$,
$$
\theta(\sigma,\tau)_i = \prod_{1 \le j \le i} \sigma_j \tau_j
$$
and
$$
W(\sigma, \tau)_a^b(P_0, \dots, P_n)
=
\frac12 \sum_{0 \le i \le n-1} 
 \theta(\sigma,\tau)_i  \cdot
\left(\sign\Big(\frac{P_{i+1}}{P_{i}}, b\Big)- \sign\Big(\frac{P_{i+1}}{P_{i}}, a\Big)\right).
$$
\end{defn}

Using the ideas of 
the proof of \cite[Theorem 3.11]{Eis}, 
we obtain the following result for special $(\sigma, \tau)$-chains. 
Note that no assumption on $a$ and $b$ is made. 

\begin{proposition}\label{prop:Cauchy_Ind_eval}
Let 
$a, b \in \R$ with $a < b$, 
$n \in \Z_{\ge 1}$ and  $\sigma, \tau \in \{-1, 1\}^{n-1}$. 
If $(P_0, \dots, P_n)$ in $\R[X]\setminus \{ 0\}$ is a special
$(\sigma, \tau)$-chain 
then
$$
\Ind_a^b\Big(\frac{P_1}{P_0}\Big) + 
\theta(\sigma, \tau)_{n-1} \cdot \Ind_a^b\Big(\frac{P_{n-1}}{P_n}\Big)
= W(\sigma, \tau)_a^b(P_0, \dots, P_n).
$$
\end{proposition}

\begin{proof}{Proof:}
We proceed by induction in $n$. If $n = 1$, the result follows from Proposition \ref{prop:inv_formula} (Inversion Formula). 

Suppose now that $n \ge 2$. 
Taking $a_1, B_1, c_1$ as in Definition \ref{defn:special_chain}, 
by Lemmas \ref{lem:mult_constant} and \ref{lem:reduce} we have
$$\begin{array}{rl}
& \displaystyle{\Ind_a^b\Big(\frac{P_0}{P_1}\Big) + \sigma_1 \cdot \tau_1 \cdot 
\Ind_a^b\Big(\frac{P_2}{P_1}\Big)} \\[5mm]
= & 
\displaystyle{\Ind_a^b\Big(\frac{-a_1P_2 - B_1P_1}{c_1P_1}\Big) + \sigma_1 \cdot \tau_1 \cdot 
\Ind_a^b\Big(\frac{P_2}{P_1}\Big)} \\[5mm]
= & 
-\sign(a_1) \cdot \sign(c_1) \cdot \displaystyle{\Ind_a^b\Big(\frac{P_2}{P_1}\Big) + \sigma_1 \cdot \tau_1 \cdot 
\Ind_a^b\Big(\frac{P_2}{P_1}\Big)} \\[5mm]
= & 0.
\end{array}
$$
We consider $\sigma' = (\sigma_2, \dots, \sigma_{n-1})$, 
$\tau' = (\tau_2, \dots, \tau_{n-1})$
and we apply the inductive hypothesis to the special $(\sigma', \tau')$-chain 
$(P_1, \dots, P_n)$. 
For $1 \le i \le n-1$ we have that $\theta(\sigma, \tau)_i = \sigma_1 \cdot \tau_1 
\cdot \theta(\sigma', \tau')_{i-1}$. 
Finally, using Proposition \ref{prop:inv_formula} (Inversion Formula) and the inductive hypothesis,
$$\begin{array}{rl}
&
\displaystyle{\Ind_a^b\Big(\frac{P_1}{P_0}\Big) + 
\theta(\sigma, \tau)_{n-1} \cdot \Ind_a^b\Big(\frac{P_{n-1}}{P_n}\Big)}
\\[5mm]
 = &
\displaystyle{\Ind_a^b\Big(\frac{P_1}{P_0}\Big) + 
\Ind_a^b\Big(\frac{P_0}{P_1}\Big) + 
\sigma_1 \cdot \tau_1 \cdot \Ind_a^b\Big(\frac{P_2}{P_1}\Big) +
\sigma_1 \cdot \tau_1 \cdot \theta(\sigma', \tau')_{n-2} \cdot \Ind_a^b\Big(\frac{P_{n-1}}{P_n}\Big)}
\\[5mm]
 = &
\displaystyle{- \frac12\sign\Big(\frac{P_1}{P_0}, a\Big) + \frac12 \sign\Big(\frac{P_1}{P_0}, b\Big)
+ \sigma_1 \cdot \tau_1 \cdot W(\sigma', \tau')_a^b(P_1, \dots, P_n)}\\[5mm]
=& \displaystyle{W(\sigma, \tau)_a^b(P_0, \dots, P_n)}
\end{array}
$$
as we wanted to prove.
\end{proof}

\begin{corollary}\label{cor:main_CI}
Let 
$a, b \in \R$ with $a < b$, 
$n \in \Z_{\ge 1}$ and  $\sigma, \tau \in \{-1, 1\}^{n-1}$. 
If $(P_0, \dots, P_n)$ in $\R[X]\setminus \{ 0\}$ is a special
$(\sigma, \tau)$-chain 
and $P_n$ divides $P_{n-1}$,
then
$$
\Ind_a^b\Big(\frac{P_1}{P_0}\Big) = W(\sigma, \tau)_a^b(P_0, \dots, P_n).
$$
 
\end{corollary}

As mentioned in the Introduction, Theorem \ref{thm:eiser}
can be deduced from Corollary \ref{cor:main_CI}.

\begin{proof}{Proof of Theorem \ref{thm:eiser} :}
Theorem \ref{thm:eiser} is a special case of Corollary
\ref{cor:main_CI} 
taking $\sigma = (1,\ldots,1)$ and $\tau = (1,\ldots,1)$, 
since the Sturm sequence is a special $((1,\ldots,1),(1,\ldots,1))$-chain and
$S_s$ divides $S_{s-1}$.
\end{proof}

\section{Proof of the main results }\label{sec:proofmain}

We fix the notation we will use from this point.

\begin{notn}\label{notn:final}
Let $P, Q
\in \R[X] \setminus \{0\}$ 
with $\deg P
 = p \ge 1$ and $\deg Q
  = q < p$.
Let 
$(d_0,\ldots,d_s)$ be
the sequence of degrees of the non-defective subresultant polynomials of $P$ and 
$Q$ in decreasing order
and let 
$d_{-1} = p+1$.

\begin{itemize}
\item Using the notation from Theorem \ref{thm:structure_thm_subresultants}, for $1 \le i \le s-1$, let 
$$
\begin{array}{rcl}
a_i & = &  t_{i-1} \cdot 
{\sRes}_{d_{i-1}}(P,Q)  \ \in \R,
\\[2mm]
B_i & = & -{\rm Quot}
\left(
t_{i} \cdot 
{\sRes}_{d_{i}}(P,Q)
\cdot T_{i-1}
,
T_i 
\right) \
 \in \R[X],
\\[2mm]
c_i & = & t_{i} \cdot 
{\sRes}_{d_{i}}(P,Q)  \
 \in \R.
\cr
\end{array}
$$

\item For $1 \le i \le s-1$, let 
$$
\begin{array}{rcl}
\sigma_i & = & \sign(a_i)  \ \in \{-1,1\}, \\[2mm]
\tau_i & = & \sign(c_i) \ \in \{-1,1\},\\[2mm]
\end{array}
$$
and let $\sigma = (\sigma_1, \dots, \sigma_{s-1})$ 
and 
$\tau = (\tau_1, \dots, \tau_{s-1})$.
\end{itemize}
\end{notn}

\begin{lemma}
 $(T_0, \dots, T_s)$ is a special $(\sigma, \tau)$-chain.
In addition, $T_s$ divides
all its elements.
\end{lemma}

\begin{proof}{Proof:}
Recall that $T_0 = P$ and $T_1 = Q$. Also,  by the Structure Theorem of Subresultants
(Theorem \ref{thm:structure_thm_subresultants}), 
we have that for $1 \le i \le s-1$,
$$
a_iT_{i+1} + B_iT_i + c_iT_{i-1} = 0.
$$
The claim follows from the definition of $\sigma,\tau$.
\end{proof}

The following lemma explores the relation between the signs of the leading coefficients of the
subresultants polynomials.

\begin{lemma}\label{lem:sign_sres} Let $P, Q
\in \R[X] \setminus \{0\}$ 
with $\deg P
 = p \ge 1$ and $\deg Q
  = q < p$. Following Notation \ref{notn:epsilon} and \ref{notn:lead_non_def},
for $0 \le i \le s$, 
$$
{\rm sign}({\rm sRes}_{d_i}(P,Q))
= \epsilon_{d_{{\rm p}(i) - 1} - d_i} \cdot {\rm sign}(t_{{\rm p}(i)}).
$$
\end{lemma}

\begin{proof}{Proof:}
For $i = 0$ the result is clear. 
For $1 \le i \le s$, by the Structure Theorem of Subresultants (Theorem \ref{thm:structure_thm_subresultants}), 
$$
{\rm sign}({\rm sRes}_{d_i}(P,Q))
=
\epsilon_{d_{i-1} - d_i} \cdot {\rm sign}(t_{i})^{d_{i-1} - d_i} \cdot 
{\rm sign}({\rm sRes}_{d_{i-1}}(P,Q))^{d_{i-1} - d_i-1}.
$$
We proceed then by induction on $i - {\rm p}(i)$. If $i = {\rm p}(i)$, then $d_{i-1} - d_i$ is odd and 
$$
{\rm sign}({\rm sRes}_{d_i}(P,Q))
=
\epsilon_{d_{{\rm p}(i) - 1} - d_i} \cdot {\rm sign}(t_{{\rm p}(i)}).
$$
If $i > {\rm p}(i)$, then $d_{i-1} - d_i$ is even, ${\rm p}(i) = {\rm p}(i-1)$ and $i-1 - {\rm p}(i-1) < i -{\rm p}(i)$; 
therefore by the inductive hypothesis,
$$
{\rm sign}({\rm sRes}_{d_i}(P,Q))
=
\epsilon_{d_{i-1} - d_i} \cdot {\rm sign}({\rm sRes}_{d_{i-1}}(P,Q)) = 
\epsilon_{d_{i-1} - d_i} \cdot  \epsilon_{d_{{\rm p}(i-1) - 1} - d_{i-1}} \cdot {\rm sign}(t_{{\rm p}(i-1)}) =
$$
$$
=
\epsilon_{d_{i-1} - d_i} \cdot  \epsilon_{d_{{\rm p}(i) - 1} - d_{i-1}} \cdot {\rm sign}(t_{{\rm p}(i)}) = 
\epsilon_{d_{{\rm p}(i) - 1} - d_i} \cdot {\rm sign}(t_{{\rm p}(i)})
$$
using equation (\ref{rem:epsilon_bis}).
\end{proof}

Now we are ready to prove Theorem \ref{thm:main_thm}.

\begin{proof}{Proof of Theorem \ref{thm:main_thm}:}
By Corollary \ref{cor:main_CI}, since
$(T_0, \dots, T_s)$ is a special $(\sigma, \tau)$-chain and $T_s$ divides $T_{s-1}$, 
$$
\Ind_a^b\Big(\frac{Q}{P}\Big)  
= W(\sigma, \tau)_a^b(T_0, \dots, T_s) =
\frac{1}{2}\sum_{0 \le i \le s-1}\theta(\sigma, \tau)_i \cdot
\left(
\sign\Big( \frac{T_{i+1}}{T_{i}}, b \Big) 
-
\sign\Big( \frac{T_{i+1}}{T_{i}}, a \Big) 
\right).
$$
So, we only need to prove that for $0 \le i \le s-1$, 
$$
\theta(\sigma, \tau)_i = \epsilon_{d_{{\rm p}(i) - 1} - d_i}\cdot 
{\rm sign}(t_{{\rm p}(i)}) \cdot {\rm sign}(t_i).
$$
Indeed, using Lemma \ref{lem:sign_sres}, 
$$\begin{array}{rcl}
\displaystyle{\theta(\sigma, \tau)_i} &= &
\displaystyle{\prod_{1 \le j \le i}\sigma_j\cdot \tau_j} \\[5mm]
&= &
\displaystyle{\prod_{1 \le j \le i}
{\rm sign}(t_{j-1}) \cdot {\rm sign}({\rm sRes}_{d_{j-1}}(P,Q)) 
\cdot {\rm sign}(t_{j}) \cdot {\rm sign}({\rm sRes}_{d_j}(P,Q))}\\[5mm]
&=& \displaystyle{{\rm sign}({\rm sRes}_{d_i}(P,Q))\cdot{\rm sign}(t_{i})}\\[5mm]
&=&\displaystyle{\epsilon_{d_{{\rm p}(i) - 1} - d_{i}} \cdot {\rm sign}(t_{{\rm p}(i)}) \cdot 
{\rm sign}(t_{i})}
\end{array}
$$
and we are done. 
\end{proof}

From Theorem \ref{thm:main_thm}, we can easily deduce Theorem \ref{thm:main_thm_variations} 
as follows. 

\begin{proof}{Proof of Theorem \ref{thm:main_thm_variations}:}
Theorem  \ref{thm:structure_thm_subresultants} implies that, if 
for some $0 \le i \le s-1$, two consecutive polynomials
 $T_i$ and $T_{i+1}$ in the sequence $(T_0,  \dots, T_s)$ have a common root $x$, 
then every polynomial in this sequence has $x$ as a root. So, suppose now that 
$a$ and $b$ are not common roots of $T_0 = P$ and $T_1 = Q$, therefore they are not common 
roots of $T_i$ and $T_{i+1}$ for any $0 \le i \le s-1$.

The proof is finished using the formula from Theorem \ref{thm:main_thm} for the Cauchy index
$\displaystyle{\Ind_a^b\Big(\frac{Q}{P}\Big)}$ and the identity $(\ref{signvar})$.
\end{proof}

Finally, we prove Theorem \ref{thm:ind_all_R}.

\begin{proof}{Proof or Theorem \ref{thm:ind_all_R}:}
We introduce the notation
$$\begin{array}{rcl}
{\rm Var}_{+\infty}(P,Q) &=&
\frac12\Big|
\sign({\rm lc}(P)) - \sign({\rm lc}(Q)) 
\Big|
,\\[3mm]
{\rm Var}_{-\infty}(P,Q) &=&
\frac12
\Big|
(-1)^{\deg (P)}\sign({\rm lc}(P)) - (-1)^{\deg (Q)}\sign({\rm lc}(Q)) 
\Big|,\\[3mm]
{\rm Var}_{-\infty}^{+\infty}(P,Q)& =&{\rm Var}_{-\infty}(P,Q)
-{\rm Var}_{+\infty}(P,Q).
\end{array}
$$
Note that, if $\deg(P) - \deg(Q)$ is even, then 
${\rm Var}_{-\infty}^{+\infty}(P,Q) = 0$, and if 
$\deg(P) - \deg(Q)$ is odd, then 
${\rm Var}_{-\infty}^{+\infty}(P,Q) = \sign({\rm lc}(P))\cdot \sign({\rm lc}(Q))$.

Choosing $r \in \R$ big enough and applying Theorem
\ref{thm:main_thm_variations}, 
$$\begin{array}{rcl}
\displaystyle{\Ind_{\R}\Big(\frac{Q}{P}\Big)} &=& 
\displaystyle{\Ind_{-r}^r\Big(\frac{Q}{P}\Big)} \\[6mm]
&=&
\displaystyle{\sum_{0 \le i \le s-1} 
\epsilon_{d_{{\rm p}(i) - 1} - d_i}\cdot {\rm sign}(t_{{\rm p}(i)}) \cdot {\rm sign}(t_i) \cdot
{\rm Var}_{-r}^{r}(T_i, T_{i+1})} \\[6mm]
&=&
\displaystyle{\sum_{0 \le i \le s-1} 
\epsilon_{d_{{\rm p}(i) - 1} - d_i}\cdot {\rm sign}(t_{{\rm p}(i)}) \cdot {\rm sign}(t_i) \cdot
{\rm Var}_{-\infty}^{+\infty}(T_i, T_{i+1})}\\[6mm]
&=&
\displaystyle{\sum_{0 \le i \le s-1, \atop  d_{i} - d_{i+1} {\scriptsize\hbox{odd}}} 
\epsilon_{d_{{\rm p}(i) - 1} - d_i}\cdot {\rm sign}(t_{{\rm p}(i)}) 
\cdot {\rm sign}(t_i) \cdot {\rm sign}({\rm lc}(T_i))\cdot {\rm sign}({\rm lc}(T_{i+1}))}.
\end{array}
$$
From this identity the result can be easily proved, taking into account that for $i \ge 1$, 
$t_i = {\rm lc}(T_i)$, but  there is an ad-hoc definition of $t_0 = 1$ (and not as the leading coefficient
of $T_0 = P$). 
\end{proof}

\textbf{Acknowledgements:} We are thankful to the anonymous referees for their helpful remarks and 
suggestions.

\end{document}